\documentclass[11pt,reqno]{amsart}
\usepackage[dvips]{graphicx}
\usepackage{epsfig}
\usepackage{amssymb}
\usepackage[usenames, dvipsnames]{color}
\usepackage{hyperref}
\usepackage{verbatim}
\usepackage[normalem]{ulem}

\theoremstyle{plain}
\newtheorem{theorem}{Theorem}[section]
\newtheorem{lemma}[theorem]{Lemma}
\newtheorem{corollary}[theorem]{Corollary}

\theoremstyle{definition}
\newtheorem{definition}[theorem]{Definition}

\theoremstyle{remark}
\newtheorem{remark}[theorem]{Remark}

\newcommand{\norm}[1]{\left\| #1 \right\|}

\def\e{\epsilon}

\def\pa{\partial}

\def\mb{\mathbb}

\def\O{\Omega}

\def\be{\begin{equation}}
\def\ee{\end{equation}}
\def\bes{\begin{equation*}}
\def\ees{\end{equation*}}
\def\bali{\begin{aligned}}
\def\eali{\end{aligned}}
\def\al{\begin{aligned}}
\def\eal{\end{aligned}}

\def\lab{\label}

\def\2O{\underline{\O}}

\def\pa{\partial}

\numberwithin{equation}{section}

\makeatletter
\def\dashint{\operatorname%
{\,\,\text{\bf--}\kern-.98em\DOTSI\intop\ilimits@\!\!}}
\makeatother

\begin{document}

\title[Relative volume comparison theorem under Kato type conditions]{Relative volume comparison under Kato type conditions for the Ricci curvature}

 \author[Wu]{Yuting Wu}

 \address{College of Mathematics and Statistics, Northwest Normal University, Lanzhou 730070, Gansu, P. R. of China, E-mail address: 090094@nwnu.edu.cn.}

\date{}

\begin{abstract}
Let $(M, g)$ be an $n$ $(\ge 3)$ dimensional, non-collapsed compact Riemannian manifold and $\operatorname{Ric}^-$ be the negative part of the Ricci curvature and $\beta \in (\frac{2n}{n+2}, 2)$. We prove a relative volume comparison theorem when $|\operatorname{Ric}^-|^\beta$ is in the Kato class (cf. Definition \ref{Ka}), which results from a new integral Laplace comparison theorem in the spirit of \cite{PW} for a suitable conformal metric. This partly addresses an expectation in \cite{TZZZZ}, where the same result was proven when $|\operatorname{Ric}^-|^2$ is in the class. 
 
\end{abstract}

\maketitle

\vspace{0.5cm}

\section{Introduction}

Several decades ago, 
the Bishop-Gromov volume comparison theorem was established under the condition that the Ricci curvature is bounded below by a constant, and it has played a fundamental role in differential geometry and the theory of manifold limits. For broader applications, it is desirable to relax the pointwise lower bound assumption on the Ricci curvature. In this direction,
in 1997, Petersen and Wei \cite{PW} developed an integral Laplace comparison theory, in which the negative part of the Ricci curvature  is assumed to be in $L^p$ ($p > n/2$). 
However, there are situations where this broader $L^p$ bound is not available.  One example is the study of K\"ahler geometry and Ricci flows. 
In 2015, Tian and Zhang \cite{TZ},
while investigating compactness problems for Fano manifolds and the K\"ahler Ricci flow, 
established a uniform $K_2$ Kato type estimate for the Hessian of the Ricci potential function (including the Ricci tensor).
More precisely, they obtained a $K_\beta$ Kato type estimate of the following form with $\beta=2$ for $F=|\operatorname{Hess}L|$, where $L$ is the Ricci potential:
\begin{definition} \lab{Ka}
Let $(M,g)$ be a compact Riemannian manifold with dimension $n\ge3$.
Assume that $F$ is an integrable function on $M$,  $d(x,y)$ is the distance between $x$ and $y$,
and $dg$ is the Riemannian volume element.
We say $F$ satisfies a $K_\beta$ Kato type condition if
there exists a constant
$C_0=C_0(\beta)>0$
such that
\begin{equation}
 \lab{kato}
K\left(|F(y)|^\beta\right)
:=
\sup_{x\in M}
\int_M
\frac{|F(y)|^\beta}
{d(x,y)^{n-2}}
\,dg(y)
\le C_0<\infty,
\end{equation}
where $0<\beta\le2$.
\end{definition}
This type of condition is, in general, weaker than the $L^p$ integrability assumption with $p>n/2$, and it is not contained in the $L^p$ framework.

In 2018, combining the $K_1$ Kato type condition with the heat kernel upper bound,
Zhang and Zhu deduced a Li-Yau gradient estimate in \cite[Remark 1.3]{ZqZm}.
Subsequently, Rose \cite{Rose} introduced a Kato type condition on the Ricci curvature defined in terms of 
the $L^1_t L^\infty_x$ norm involving the heat semigroup. 
Under a suitable smallness assumption on this quantity, he established a Li-Yau-type gradient estimate, 
which in turn yielded Harnack inequalities and heat kernel upper bounds.

In 2019, Carron \cite{Carron} considered a related Kato type class using convolutions with the heat kernel.
His condition is essentially the same as the above with $\beta=1$ if the manifold is non-collapsed and the heat kernel has a Gaussian lower bound, but is more natural otherwise.
This type of Kato class was first introduced in \cite{Sturm} in the Euclidean case and extended to the manifold case in \cite{Zhang1}.
Under a suitable smallness condition of his Kato type condition for the Ricci curvature, Carron established Li-Yau-type gradient estimates and heat kernel estimates and derived several geometric and analytic consequences, including Sobolev inequalities, volume doubling, and Poincar\'e inequalities.
Recently, Carron, Mondello, and Tewodrose \cite{CMT1} established a partial monotonicity formula for the volume ratio involving a Kato-dependent error term by using a time-change technique (see Remark \ref{rem} below).

In 2025, Tian et al. \cite{TZZZZ}, 
by the new method of constructing suitable conformal metrics, together with the $K_2$ Kato type Ricci curvature bound along the K\"ahler Ricci flow, established the necessary integral Laplace comparison result in the spirit of \cite{PW}, which was a main obstacle in the convergence theory. Consequently, they 
extended  parts of Cheeger-Colding theory to the K\"ahler Ricci flow setting. 
As an application, they obtained a direct and quantitative proof of the Hamilton-Tian conjecture, which also has direct implications to the Yau-Tian-Donaldson conjecture (cf. \cite[Section 6]{TZz}).

Tian et al. (cf. Subsection 1.2 in \cite{TZZZZ}, after Theorem 1.5) expected that the exponent $2$ in the $K_2$ Kato condition could be reduced, but did not specify an explicit lower exponent.
In this paper, we show that the exponent $2$ can be explicitly reduced to $\beta \in (\frac{2n}{n+2}, 2)$, while preserving the corresponding integral Laplace comparison, relative volume comparison theorem.
This provides a weaker curvature assumption than the original $K_2$ Kato condition.  
 

Compared with \cite{TZZZZ}, the key new ingredient is to exploit the Kato condition to derive a Morrey-type estimate, and then use an Adams-type estimate for the associated Riesz potential in Morrey spaces to obtain the required regularity estimates for the conformal function. 
This provides a crucial input for the proof of the theorem.

Before stating the main results, we fix some notations. 
Throughout the paper we use $(M, g)$ to denote a compact Riemannian manifold with dimension $n$ ($n\ge3$)
with a metric $g$. The $\operatorname{Ric}$ denotes the Ricci curvature, and $\operatorname{Ric}^-$ denotes its negative part. 
Let $\operatorname{diam} M$ be the diameter of $M$.
Let $B(x, r)$ be the geodesic ball centered at $x$ with radius $r$, and $\operatorname{Vol}\bigl(B(x, r)\bigr) $ denote its volume.
We use $\Delta$, $\nabla$, $\operatorname{Hess}$, and $\langle \cdot, \cdot \rangle$ to denote the Laplace operator, gradient, Hessian, and inner product with respect to $g$, respectively. 
For a conformal metric $\tilde g$ to be defined later,
the corresponding operators with respect to $\tilde{g}$ are denoted by $\Delta_{\tilde{g}}$, $\nabla_{\tilde{g}}$, $\operatorname{Hess}_{\tilde{g}}$, and $\langle \cdot, \cdot \rangle_{\tilde{g}}$. 
The volume elements for $g$ and $\tilde{g}$ are written as $dg$ and $d\tilde{g}$.
Given a tensor field $X$ on $M$, the notations $|X|$ and $|X|_{\tilde{g}}$ are used for their norms under $g$ and $\tilde g$, respectively. For example, if $X$ is a vector field, then $|X|^2=\left\langle X, X\right\rangle $ and
$|X|^2_{\tilde{g}}=\left\langle X, X\right\rangle _{\tilde{g}}$.

The main result is as follows:

\begin{theorem}\label{th:1.1}
Let $(M, g)$ be a compact Riemannian manifold with dimension $n$ ($n\ge3$),
and set
$
\beta \in (\frac{2n}{n+2}, 2).
$
Assume that the following conditions hold:
\begin{enumerate}
\item[(i).] the $K_\beta$ Kato type bound for the negative part of the Ricci curvature \eqref{kato} holds,
i.e. $K(|\operatorname{Ric}^-|^\beta)\leq C_0$;

\item[(ii).] $\operatorname{diam}M<\infty$;

 \item[(iii).] the Sobolev inequality holds with a constant $C_S$:
\begin{equation*}
	\left(
	\int_M \psi^{\frac{2n}{n-2}}\,dg
	\right)^{\frac{n-2}{n}}
	\leq
	C_S\left( \int_M |\nabla\psi|^2\,dg
	+
	\int_M \psi^2\,dg \right) 
\end{equation*}
for all $\psi\in C^\infty(M).$

\end{enumerate}
Then there exists a smooth function $h$ and positive constants $b_1, b_2$ and $b_3$ depending
only on the parameters $C_0,$ $C_S,$ $\operatorname{diam}M,$ $\beta$ and $n$, 
and $C^*_p$, $H^*_q$ which also depend on $p$, $q$, respectively,
such that the following results hold.

\begin{enumerate}	
\item[(a).] $b_1 \le h \le b_2$;
	
\item[(b).] $\Vert h \Vert_{C^{\frac{2(\beta-1)}{3\beta-2}}(M)} \le b_3$,
\quad 	$\|\nabla h\|_{L^p(M)}\le C^*_p\quad \forall p\in(n,\frac{2\beta}{2-\beta})$;

\item[(c).] Let $\tilde{g}  = h g$ and $q\in(n,\,\frac{4n\beta}{2n-(n-2)\beta})$. For any point $x_0 \in M$, let $\tilde r =d_{\tilde g}(x_0, x)$ be the geodesic distance from $x_0$ to $x$ under $\tilde g$. Let
\[
\psi = \Delta_{\tilde g} \tilde r - \frac{n-1}{\tilde r}
\]defined outside the cut-locus of $x_0$.
Then the positive part of $\psi$ denoted by $\psi_+$, satisfies the Petersen-Wei bound:
\[
\Vert \psi_+ \Vert_{L^q(M, \tilde g)} = \left(\int_M \psi^q_+ d\tilde g \right)^{\frac{1}{q}} \le H^*_q.
\]
\end{enumerate}
\end{theorem}

Condition (iii) is satisfied by many manifolds under mild conditions (See \cite[Chapter 5]{Gri} and \cite[Chapter 3]{Lau}).
It is also satisfied by a compact Ricci flow in a finite time interval with bounded scalar curvature (cf. \cite[Chapter 6]{Zhang}).

An immediate consequence of Theorem \ref{th:1.1} is the classical volume comparison theorem for both the original and conformal metrics, modulo a small error term which goes to $0$ as the scale goes to $0$ at a definite rate.

\begin{corollary}\lab{volcomp} 
The following conclusions hold:	
\begin{itemize}		
\item[(a).] Given $q\in(n,\,\frac{4n\beta}{2n-(n-2)\beta})$, there exists a constant $\tilde  b_q$ depending only on $q$ and the basic parameters $C_0$, $\beta$, $\operatorname{diam} M$, $C_S$ and $n$ in Theorem \ref{th:1.1} such that
	\[
	\frac{d}{dr}\left(\frac{\operatorname{Vol}_{\tilde  g}(B_{\tilde  g}(x,r))}{r^n}\right)
	\le \frac{\tilde b_q}{r^{\frac{n}{q}}},
	\]
	\[
	\frac{\operatorname{Vol}_{\tilde  g}(\partial B_{\tilde  g}(x, r_2))}{r_2^{n-1}} -
	\frac{\operatorname{Vol}_{\tilde  g}(\partial B_{\tilde  g}(x, r_1))}{r_1^{n-1}}
	\le \tilde  b_q \left(r^{\frac{q-n}{q-1}}_2 - r^{\frac{q-n}{q-1}}_1 \right)^\frac{q-1}{q}.
	\]
	
	\item[(b).] Given $q\in(n,\,\frac{4n\beta}{2n-(n-2)\beta})$, there exists a constant $b_q$ such that 
	\[
	\frac{\operatorname{Vol}_{g}(B(x, r_2))}{r^n_2} - \frac{\operatorname{Vol}_{g}(B(x, r_1))}{r^n_1} 
	\le b_q \operatorname{min} \left\lbrace  r_2^{1-\frac{n}{q}}, \left(  1 - \frac{r_1}{r_2} \right) ^\delta\right\rbrace  , \qquad r_2>r_1>0,
	\]
where $b_q$ and $\delta$ depend only on $q$ and the basic parameters $C_0$, $\beta$, $\operatorname{diam} M$, $C_S$ and $n$ in Theorem \ref{th:1.1}.	
\end{itemize}	
\end{corollary}
We mention that when $q \to n$ or $\frac{4n\beta}{2n-(n-2)\beta}$ the constants $b_q$ and $\tilde  b_q$ may go to $\infty$.

{\remark  \label{rem}
In Theorem C of \cite{CMT1}, a similar result to part (b) of the corollary is presented. However, 
the term on the right hand side does not go to 0 as $r_1 \to r_2^-$,
and $r_1$ is limited to the interval $(\eta r_2, r_2)$ for some $0<\eta<1$.
On the other hand, their result holds for all $\beta>1$, which is a weaker condition than ours.
The main point is that in small scales, the volume ratio is almost monotone, which is a key step in developing Cheeger-Colding degeneration theory for manifolds.}

\section{Proof of Theorem \ref{th:1.1}}

Under assumptions (i) and (iii) in Theorem \ref{th:1.1}, the following properties hold on $(M, g)$,
which will be used explicitly or implicitly in the proof of the theorem.

1. $(M,g)$ is non-collapsed: there exists a constant $C=C (n, \beta, C_0, C_S, \operatorname{diam}M)>0$ such that
\begin{equation} \label{non-collapsed}
	\operatorname{Vol}(B(x,r))
	\ge
	C^{-1} r^n
\end{equation}
for all $x\in M$ and $0<r\le \operatorname{diam}M$.

2. $(M,g)$ is non-inflated:
there exists a constant $C=C (n, \beta, C_0, C_S, \operatorname{diam}M)>0$ such that
\begin{equation} \label{non-inflated}
\operatorname{Vol}(B(x,r))
\leq
C r^n
\end{equation}
for all $x \in M$ and $0<r\le\operatorname{diam}M.$

3. Poincar\'e inequality:
for every geodesic ball $B=B(x,r)$ with $0<r\leq \operatorname{diam}M$ and every $\psi\in C^1(B)$,
there exists a constant $C_P=C_P (n, \beta, C_0, C_S, \operatorname{diam}M)>0$ such that
\begin{equation}
	\int_B |\psi-\psi_B|^2\,dg
	\leq
	C_P r^2
	\int_B |\nabla\psi|^2\,dg,
\end{equation}
where
$
\psi_B
=
\frac{1}{\operatorname{Vol}(B)}
\int_B\psi\,dg 
$
denotes the average of $\psi$ over $B$.

Let us explain why.
The Sobolev inequality in assumption (iii) of Theorem \ref{th:1.1} implies Property 1 as well as the heat kernel upper bound in \eqref{gulbwan}. Both are classical results that can be found in many papers and books. 
For example, see the references cited in \cite[Sections 4.1 and 4.2]{Zhang}, although \cite{Zhang} is not the first to establish them.
Under the $K_\beta$ Kato type condition with $\beta>1$ and the heat kernel upper bound,
one can derive a Li-Yau gradient estimate (cf. \cite[Remark 1.3]{ZqZm} or \cite[Section 3.6]{Carron}),
which in turn yields a lower bound for the heat kernel in \eqref{gulbwan},
and hence gives Property 2.
More precisely,
let $G=G(x, t, y)$ be the heat kernel on
$(M, g)$. Then there exist positive constants $a_1$ and $a_2$ such that
\be
\lab{gulbwan}
\frac{a_1}{t^{\frac{n}{2}}} e^{\frac{-a_{2}
		d(x, y)^2}{t}} \le 
	G(x, t, y) \le \frac{1}{a_1 t^{\frac{n}{2}}} e^{- \frac{d(x, y)^2}{a_2 t}}, \qquad t \in (0, 1]. 
\ee
In addition, the Li-Yau gradient estimate and \eqref{gulbwan} imply Property 3. 
The arguments can be found in \cite[Chapter 6]{Gri} and \cite[Chapter 5]{Lau}.
See also \cite{Rose} and \cite{Carron} for two related definition of Kato class using the heat kernel, where similar properties were proven under smallness of the $K_1$ Kato type norm. It should be pointed out that a priori the heat kernel bound is not yet known so it is harder to check Kato norms defined with heat kernels.
 
Recall that the normalized Green's function $\Gamma$ changes sign on compact manifolds. 
Nevertheless,
there exists a positive constant $C=C (n,\beta,C_0,C_S,\operatorname{diam}M)$ such that
\begin{equation}
	\lab{Green} 
	|\Gamma(x, y) | \le \frac{C}{d(x, y)^{n-2}}, \quad
	|\nabla \Gamma(x, y) | \le \frac{C}{d(x, y)^{n-1}},
\end{equation}
which can be derived from the bounds on the heat kernel and the corresponding gradient estimate as in \cite[Lemma 2.3]{TZ1}.

Next, we recall the definitions of the classical maximal function, the fractional maximal function and the Riesz potential on the manifold, which will be used in the proof of Theorem \ref{th:1.1}.

\begin{definition}[Hardy-Littlewood maximal function] \label{Mf(x)}
	For a locally integrable function \(f\) on \(M\), the Hardy-Littlewood maximal function is defined by
	\[
	(Mf)(x):=\sup_{0<r\le \operatorname{diam}M}\frac{1}{\operatorname{Vol}(B(x, r))}\int_{B(x, r)} |f(y)|\,dg(y).
	\]
\end{definition}

\begin{definition}[Fractional maximal function] \label{Malphaf(x)}
	For \(0\le \sigma <n\), the fractional maximal function of \(f\) is given by
	\[
	(M_\sigma f)(x):=\sup_{0<r\le \operatorname{diam}M} (\operatorname{Vol}(B(x, r)))^{\frac{\sigma-n}{n}}\int_{B(x, r)} |f(y)|\,dg(y).
	\]
	When \(\sigma=0\), this reduces to the classical Hardy-Littlewood maximal function \(Mf\).
\end{definition}

\begin{definition}[Riesz potential] \label{If(x)}
	For \(0<\alpha<n\), the Riesz potential (or fractional integral) of order \(\alpha\) is defined as
	\[
	(I_\alpha f)(x):=\int_M \frac{|f(y)|}{d(x,y)^{n-\alpha}}\,dg(y),
	\]
	whenever the integral converges.
\end{definition}

These operators are natural geometric extensions of the classical Euclidean operators to the setting of a Riemannian manifold. 
More precisely, their definitions are formulated intrinsically in terms of the geodesic distance $d(x,y)$, the Riemannian volume measure $dg$ and the volume of geodesic balls $B(x,r)$, thereby encoding the underlying geometry of the manifold. 

In addition, we also need the embedding result in \cite[Lemma 2.4]{TZ}.
\begin{lemma}
	\lab{leimbed}
	Let $V$ be a smooth function, $p$ be a point on $M$ and $r$ be a
	positive number such that $r \le \operatorname{diam}M/2$. Then for any smooth function $\phi$
	on $M$, the following embedding result holds
	\begin{equation}
		\al
		&\int_{B(p, r)} |V(x)| \phi^2(x) dg(x) \\
		&\le C \sup_{z \in B(p, 2r)} \int_{B(z, 2r)}
		\frac{|V(x)|}{d(z, x)^{n-2}} dg(x) \left( \Vert \nabla \phi \Vert^2_{L^2(B(p, 2r))}
		+ r^{-2} \Vert  \phi \Vert^2_{L^2(B(p, 2r))} \right).
		\eal
	\end{equation} 
\end{lemma}

We now present the

\noindent
\textbf{\emph{Proof of Theorem \ref{th:1.1}.}}
It consists of the following four steps.

\noindent{\it Step 1.} The Bochner inequality under conformal transformations.

Take a positive smooth function
$
h=e^{2f}>0
$
on $(M, g)$ and define the new conformal metric
$
\tilde g:=hg.
$
The direct computation gives the transformation rules for $\Delta_{\tilde{g}}$, $\nabla_{\tilde{g}}$, $\operatorname{Hess}_{\tilde{g}}$ of a smooth function $u$ (see \cite[pp. 8-9]{TZZZZ}). We deduce
\[
\al
\Delta ( |\nabla_{\tilde g} u |^2_{\tilde g}) &= \Delta ( h^{-1} |\nabla u |^2 ) = \Delta ( e^{- 2 f} |\nabla u |^2 )\\
&= | \nabla u|^2 \Delta e^{-2 f} + 2\left\langle \nabla | \nabla u |^2, \nabla e^{-2f}\right\rangle  + e^{-2f} \Delta |\nabla u|^2.
\eal
\]
Using the classical Bochner formula, we obtain
\begin{equation} \label{Bochner}
\begin{aligned}
	\Delta\bigl(|\nabla_{\tilde g}u|_{\tilde g}^2\bigr)
	&= |\nabla u|^2\bigl(e^{-2f}4|\nabla f|^2-2e^{-2f}\Delta f\bigr) -4e^{-2f}\langle \nabla|\nabla u|^2,\nabla f\rangle \\
	&\quad +e^{-2f}\left( 2|\operatorname{Hess}u|^2+2\langle \nabla\Delta u,\nabla u\rangle+2\operatorname{Ric}(\nabla u,\nabla u)\right) .
\end{aligned}
\end{equation}
Let $V(x)=|\operatorname{Ric}^{-}(x)|$.
Without loss of generality we can assume $V$ is smooth by choosing a smooth majorant of $|Ric^-(x)|$.
Substituting this into \eqref{Bochner} yields the following modified inequality
\begin{equation}\label{eq:bochner_mod}
	\begin{aligned}
		e^{2f}\Delta\left( |\nabla_{\tilde g}u|_{\tilde g}^2\right) 
		&\ge |\nabla u|^2\left( 4|\nabla f|^2-2\Delta f-2V\right)  -4\langle \nabla|\nabla u|^2,\nabla f\rangle \\
		&\quad +2|\operatorname{Hess}u|^2+2\langle \nabla\Delta u,\nabla u\rangle.
	\end{aligned}
\end{equation}
It is easy to see that
$
|\nabla u|^2\bigl(4|\nabla f|^2-2\Delta f-2V\bigr)
$
is the bad term, and we will deal with it in the next step.

\noindent{\it Step 2.} Uniform boundedness and regularity of  $ h=\phi^2.$

\noindent {\it Step 2.1.}  Eliminating the bad term in \eqref{eq:bochner_mod}.

From the assumption,
there exists a constant $C_0>0$ such that
\begin{equation*}
K\left(|\operatorname{Ric}^-|^\beta\right)=
\sup_{x \in M} \int_{M} \frac{|\operatorname{Ric}^-(y)|^\beta}{d^{n-2}(x, y)} dg(y) \le C_0.
\end{equation*} 
For any $0<r<\frac12\operatorname{diam}M, $
H\"older's inequality gives
\begin{align} \lab{rick1r}
	&\sup_{x\in M} \int_{B(x,r)} \frac{V(y)}{d(x,y)^{n-2}} \,dg(y) \nonumber\\ 
	&\quad\le \left( \sup_{x\in M} \int_{B(x,r)} \frac{|\operatorname{Ric}^{-}(y)|^\beta} {d(x,y)^{n-2}} \,dg(y) \right)^{\frac{1}{\beta}} \left( \sup_{x\in M} \int_{B(x,r)} \frac{1}{d(x,y)^{n-2}} \,dg(y) \right)^{1-\frac{1}{\beta}} \nonumber\\ 
	&\quad\le C_0^{\frac{1}{\beta}} \left( \sup_{x\in M} \int_{B(x,r)} d(x,y)^{2-n} \,dg(y) \right)^{1-\frac{1}{\beta}} \nonumber\\ 
	&\quad\le C C_0^{\frac{1}{\beta}} r^{2(1-\frac{1}{\beta})}.
\end{align} 
Here $C>0$ depends only on the dimension $n$ and the fixed geometric constants (independent of the scale).

Moreover, recall the embedding result in Lemma~\ref{leimbed}.
Combining \eqref{rick1r} with a simple covering argument using balls with a small but definite radius yields
\begin{equation*} 
	\label{quadrav} 
	\int_M \left( |\nabla\phi|^2-V\phi^2 \right) \,dg \ge -\rho_0 \int_M\phi^2\,dg, \qquad \forall\,\phi\in C^\infty(M)
\end{equation*}
for some constant $\rho_0>0$. 
Now, we minimize the quadratic form
$
\int_M \left(|\nabla\phi|^2 - V\phi^2\right)dg
$
subject to the normalization constraint $\|\phi\|_{L^2(M)}=1$.
Let $\rho_0\ge 0$ denote the optimal constant.
By the standard variational characterization of the lowest eigenvalue  of the operator $-\Delta-V$ and elliptic regularity, 
there exists a smooth positive function $\phi$, whose Euler-Lagrange equation is
\begin{equation} \label{eqphi} 
	\Delta\phi+V\phi-\rho_0\phi=0 \qquad\text{on }M.
\end{equation} 
Set 
$
	f:=\log\phi. 
$
It follows from \eqref{eqphi} that 
\begin{equation} \label{eqffphi} 
	\Delta f+|\nabla f|^2+V-\rho_0=0. 
\end{equation} 
Substituting \eqref{eqffphi} into \eqref{eq:bochner_mod},  we deduce 
\begin{equation}
	\lab{eq:bochner_mod1}
	\al
	e^{2 f} \Delta ( |\nabla_{\tilde g} u |^2_{\tilde g})  \ge  2 | \operatorname{Hess}  u|^2 + 2 \left\langle  \nabla \Delta u, \nabla u\right\rangle - 4 \left\langle \nabla |\nabla u|^2, \nabla f\right\rangle  - 2 \rho_0 | \nabla u|^2.
	\eal
\end{equation}
It is noteworthy that the bad term has been canceled. 

\noindent {\it Step 2.2.} The proof of Theorem \ref{th:1.1} (a) and (b).

The proof of part (a) in Theorem \ref{th:1.1} is similar to the argument 
in \cite[pp. 13-14]{TZZZZ}, so we omit the details. 
The most important part for us is the proof of part (b) of Theorem \ref{th:1.1},
which requires the new ingredient.

First we prove H\"older continuity of $\phi$.

Applying \eqref{eqphi}, 
since the mean value of $V\phi-\rho_0 \phi$ is zero,
we deduce
\[
\phi(x) = \int_M \Gamma(x, y) (V(y)-\rho_0) \phi(y) dg(y).
\]
Next, we apply a technique from \cite[pp. 423-424]{CFG}.
For any $x,x_0\in M$ such that
$
d(x, x_0)<10^{-\frac{3\beta-2}{2(\beta-1)}},
$
let $N>10$ be a large number to be determined later.
One derives
\begin{equation}
\lab{phix-x0}
\al
|\phi(x)-\phi(x_0)| 
&\le \int_M |\Gamma(x, y)-\Gamma(x_0, y)| |V(y)-\rho_0| dg(y)\\
&=\int_{d(x_0, y)>N d(x, x_0)}|\Gamma(x, y)-\Gamma(x_0, y)| |V(y)-\rho_0| dg(y)\\
&\quad + \int_{d(x_0, y) \le N d(x, x_0)} |\Gamma(x, y)-\Gamma(x_0, y)| |V(y)-\rho_0|dg(y)\\
& \equiv J_1 + J_2.
\eal
\end{equation} 
Under the condition ${d(x_0, y)>N d(x, x_0)}$, 
using \eqref{Green}, we deduce
\[
|\Gamma(x, y)-\Gamma(x_0, y)| \le  \frac{C d(x, x_0)}{d(x_0, y)^{n-1}} =
\frac{C d(x, x_0)}{d(x_0, y) \, d(x_0, y)^{n-2}}
\le  \frac{C}{N d(x_0, y)^{n-2}}.
\]
Consequently,
\begin{equation} \lab{i1c1k}
\al
J_1 &\le \frac{C}{N} \int_{d(x_0, y)>N d(x, x_0)} \frac{|V(y)-\rho_0| }{d(x_0, y)^{n-2}} dg(y)\\
&\le\frac{C}{N} \left( \int_{d(x_0, y)>N d(x, x_0)} \frac{|V(y)-\rho_0|^\beta }{d(x_0, y)^{n-2}} dg(y) \right)^{\frac{1}{\beta}}
\, \left( \int_{d(x_0, y)>N d(x, x_0)} \frac{1}{d(x_0, y)^{n-2}} dg(y) \right)^{1-\frac{1}{\beta}}\\
&\le C_1 \frac{C}{N} \left( K(|V-\rho_0|^\beta)\right) ^{\frac{1}{\beta}},
\eal
\end{equation} 
where the last step used \eqref{non-inflated}.
For the term $J_2$, we apply H\"older's inequality to get
\[
\al
J_2 
&\le C \int_{d(x_0, y) \le N d(x, x_0)}  \frac{|V(y)-\rho_0| }{d(x_0, y)^{n-2}} dg(y)
+ C \int_{d(x_0, y) \le N d(x, x_0)}  \frac{|V(y)-\rho_0| }{d(x, y)^{n-2}} dg(y)\\
& \le C \left( \int_{d(x_0, y) \le N d(x, x_0)} \frac{|V(y)-\rho_0|^\beta }{d(x_0, y)^{n-2}} dg(y) \right)^{\frac{1}{\beta}}
\, \left( \int_{d(x_0, y) \le N d(x, x_0)} \frac{1}{d(x_0, y)^{n-2}} dg(y) \right)^{1-\frac{1}{\beta}}\\
&\quad  + C \left( \int_{d(x_0, y) \le N d(x, x_0)} \frac{|V(y)-\rho_0|^\beta }{d(x, y)^{n-2}} dg(y) \right)^{\frac{1}{\beta}}
\, \left( \int_{d(x_0, y) \le N d(x, x_0)} \frac{1}{d(x, y)^{n-2}} dg(y) \right)^{1-\frac{1}{\beta}}.
\eal
\]
Again using \eqref{non-inflated} and proceeding as in \eqref{rick1r}, we obtain
\begin{equation}\lab{i2c2k}
J_2 \le C_2 C \left( K(|V-\rho_0|^\beta)\right)^{\frac{1}{\beta}}  \left( N d(x, x_0) \right) ^{2(1-\frac{1}{\beta})}.
\end{equation}
Substituting \eqref{i2c2k} and \eqref{i1c1k} into \eqref{phix-x0}, we conclude that
\[
|\phi(x)-\phi(x_0)| \le \frac{C_3}{N} + C_4 (N d(x, x_0))^{2(1-\frac{1}{\beta})}.
\]
Taking $N=d(x, x_0)^{-\frac{2(\beta-1)}{3\beta-2}}$, it follows that
\[
|\phi(x)-\phi(x_0)| \le C_5 d(x, x_0)^{\frac{2(\beta-1)}{3\beta-2}}.
\]
Thus, $\phi\in C^{\frac{2(\beta-1)}{3\beta-2}}(M).$ 

From now, we prove the estimate
$\|\nabla h\|_{L^p(M)}\le C^*_p$ of part (b).
To this end, we need the following lemma.

\begin{lemma}
	\label{lemma1.1}
	Let $(M,g)$ be as in Theorem \ref{th:1.1}. Define $W:=V-\rho_0.$
	Assume that
	$
	\beta\in(\frac{2n}{n+2}, 2)
	$
	and that there exists a constant $C_0>0$ such that
	\begin{equation}
		\label{eq:1.1}
		K(|W|^\beta)=
		\sup_{x\in M}
		\int_M
		\frac{|W(y)|^\beta}{d(x,y)^{n-2}}
		\,dg(y)
		\le C_0.
	\end{equation}
	Then
\[
I_1 W\in L^p(M),\qquad
p\in(n,\frac{2\beta}{2-\beta}).
\]
\end{lemma}

\begin{proof}[\textbf{Proof}]

For every $y\in B(x, r)$, we have
$d(x,y)\le r.$
Therefore,
$$d(x,y)^{-(n-2)}\ge r^{-(n-2)}.$$
Using (\ref{eq:1.1}), we obtain
\begin{equation*}\label{eq:1.2}
	C_0\ge
	\int_{B(x, r)}\frac{|W(y)|^\beta}{d(x,y)^{n-2}}dg(y)
	\ge r^{-(n-2)}\int_{B(x, r)}|W(y)|^\beta dg(y),
\end{equation*}
which implies
\begin{equation}\lab{MW}
\int_{B(x, r)}|W(y)|^\beta dg(y)\le C_0r^{n-2}.
\end{equation} 
This is precisely a Morrey-type estimate of codimension 2.	

Applying H\"older's inequality together with \eqref{non-inflated}, we get
\begin{equation*}
\begin{aligned}
	\int_{B(x, r)}|W(y)|\,dg(y)
	&\le
	\left(\int_{B(x, r)}|W(y)|^\beta\,dg(y)\right)^{\frac{1}{\beta}}
	\operatorname{Vol}(B(x, r))^{1-\frac{1}{\beta}} \\
	&\le C \left(C_0 r^{n-2}\right)^{\frac{1}{\beta}} \left(r^n\right)^{1-\frac{1}{\beta}}
	= C C_0^{\frac{1}{\beta}} r^{n-\frac{2}{\beta}}.
\end{aligned}
\end{equation*}
Thus,
\[
r^{\frac{2}{\beta}-n}\int_{B(x, r)}|W(y)|\,dg(y)
\le  C C_0^{\frac{1}{\beta}}.
\]
Using Property 1, we obtain the following fractional maximal estimate (see Definition \ref{Malphaf(x)})
$$
( M_{\frac{2}{\beta}} W) (x)\le C\sup_{0<r\le \operatorname{diam}M}r^{\frac{2}{\beta}-n}\int_{B(x, r)}|W(y)|\,dg(y)\le C C_0^{\frac{1}{\beta}}.
$$
Consequently,
\begin{equation}\lab{MW1}
\| M_{\frac{2}{\beta}} W \|_{L^\infty(M)} \le C C_0^{\frac{1}{\beta}}.
\end{equation}

Next, we use
Adams' estimate for Riesz potentials on Morrey spaces (see \cite{Adams}) which was originally established in the Euclidean setting.
Under assumptions of Theorem \ref{th:1.1}, the result can be obtained similarly on manifolds, and we outline the proof for completeness.
Let $R>0$ be a number to be determined later.
Fix \(x\in M\),
we split the Riesz potential (see Definition \ref{If(x)}) into near and far parts, i.e.
\[
(I_\alpha W)(x)=\int_{d(x,y)<R} \frac{|W(y)|}{d(x,y)^{n-\alpha}}\,dg(y)
+\int_{d(x,y)\ge R} \frac{|W(y)|}{d(x,y)^{n-\alpha}}\,dg(y)
=I_{\rm near}+I_{\rm far}.
\]

Decompose $d(x,y)<R$ into  the dyadic annuli
\[
E_{-k}(x)=\{2^{-k}R\le d(x,y)<2^{-k+1}R\},\qquad k=1,2,\dots
\]
On $E_{-k}$, we have
$
d(x,y)^{-(n-\alpha)}\le C(2^{-k}R)^{-(n-\alpha)}.
$
Thus,
\begin{equation*} \lab{I_near}
	\al 
	I_{\rm near}
	&=\int_{d(x,y)<R}\frac{|W(y)|}{d(x,y)^{n-\alpha}}dg(y)
	\le\sum_{k=1}^\infty\int_{E_{-k}}\frac{|W(y)|}{d(x,y)^{n-\alpha}}dg(y)\\
	&\le\sum_{k=1}^\infty C(2^{-k}R)^{-(n-\alpha)}\int_{B(x, \,2^{-k+1}R)} |W(y)|dg(y).\\
	\eal
\end{equation*}
By the definition of the maximal function (see Definition \ref{Mf(x)}), we obtain
\[
\int_{B(x, \,2^{-k+1}R)} |W(y)| dg(y) \le \operatorname{Vol}(B(x, \,2^{-k+1}R)) (MW)(x).
\]
Using the volume estimate
$
\operatorname{Vol}(B(x, \,2^{-k+1}R))\le C(2^{-k+1}R)^n,
$
we conclude that
\begin{equation} \lab{I_near2}
	\al 
	I_{\rm near}
	&\le C\sum_{k=1}^\infty (2^{-k-1}R)^{-(n-\alpha)}(2^{-k+1}R)^n (MW)(x)\\
	&\le \sum_{k=1}^\infty2^{-k\alpha} R^\alpha (MW)(x)
	\le C_1 R^\alpha (MW)(x).\\
	\eal
\end{equation}
Moreover, we decompose $d(x,y)\ge R$ into the dyadic annuli
\[
E_k=\{2^k R\le d(x,y)<2^{k+1}R\}.
\]
Since $M$ is compact, $E_k=\varnothing$ whenever $2^kR>\operatorname{diam} M$. Therefore, only finitely many of the sets $E_k$ are nonempty.
On \(E_k\), we get
\[
d(x,y)^{-(n-\alpha)}\le (2^kR)^{-(n-\alpha)}.
\]
According to the definition of the fractional maximal function (see Definition \ref{Malphaf(x)}), we deduce
\begin{equation*} 
	\al 
	I_{\rm far}
	&=
	\int_{d(x,y)\ge R}\frac{|W(y)|}{d(x,y)^{n-\alpha}}\,dg(y)
	\le \sum_{k=0}^\infty\int_{E_k}\frac{|W(y)|}{d(x,y)^{n-\alpha}}\,dg(y)\\
	&\le\sum_{k=0}^\infty (2^kR)^{-(n-\alpha)}\int_{B(x, 2^{k+1}R)}|W(y)|\,dg(y)\\
	&\le\sum_{k=0}^\infty (2^kR)^{-(n-\alpha)}(2^{k+1}R)^{(n-\frac{2}{\beta})}
	(M_{\frac{2}{\beta}}W)(x)\\
	&\le C R^{\alpha - \frac{2}{\beta}} (M_{\frac{2}{\beta}} W)(x) \sum_{k=0}^{\infty} 2^{k(\alpha - \frac{2}{\beta})}.\\
	\eal
\end{equation*}
Since $\alpha < \frac{2}{\beta}$, the series satisfies
$
\sum_{k=0}^{\infty} 2^{k(\alpha - \frac{2}{\beta})} \le C.
$
Hence,
\begin{equation}  \lab{I_far}
I_{\rm far}\le C_2 R^{\alpha-\frac{2}{\beta}}(M_{\frac{2}{\beta}}W)(x).
\end{equation}
Combining \eqref{I_near2} and \eqref{I_far}, for every $0<R\le \operatorname{diam}M$,
we get
\[
(I_\alpha W)(x)\le C_1 R^\alpha (MW)(x)+C_2 R^{\alpha-\frac{2}{\beta}}(M_{\frac{2}{\beta}}W)(x),
\]
where $0<\alpha<\frac{2}{\beta}$.
Without loss of generality, we can assume that $|W|$ is not identically zero. 
Define
\[
R_0
:=
\left[ \frac{(M_{\frac{2}{\beta}}W)(x)}{(MW)(x)}\right] ^{\frac{\beta}{2}}.
\]
We can find that
$R_0\le C \operatorname{diam}M$ (from Definition \ref{Malphaf(x)}).
Choosing $R=\frac{R_0}{C}$,
we conclude that, for every \(x\in M\),
\begin{equation*}
	(I_\alpha W)(x)
	\le
	C\big[(MW)(x)\big]^{1-\frac{\alpha\beta}{2}}
	\left[
	\left(M_{\frac{2}{\beta}}W\right)(x)
	\right]^{\frac{\alpha\beta}{2}}.
\end{equation*}
Applying H\"older's inequality gives
\[
\norm{(I_\alpha W)(\cdot)} _{L^p(M)}
\le C\norm{(MW)(\cdot)} ^{1-\frac{\beta\alpha}{2}} _{L^\beta(M)}
\norm{(M_{\frac{2}{\beta}}W)(\cdot)} ^{\frac{\beta\alpha}{2}} _{L^\infty(M)}
\le  C C_0^{\frac{1}{\beta}},
\]
where $\frac{1}{p} = \frac{1}{\beta} - \frac{\alpha}{2} $.
Here we used \eqref{MW} and \eqref{MW1}.

Setting $\alpha=1$ and using the condition 
$
	\beta>\frac{2n}{n+2},
$
we deduce that
$
\frac{1}{p} = \frac{1}{\beta} - \frac{1}{2},
$
i.e.
$$p=\frac{2\beta}{2-\beta}>n.$$ 
Hence,
\[
I_1 W\in L^p(M),\qquad
p\in(n,\frac{2\beta}{2-\beta}).
\]
This completes the proof of  Lemma~\ref{lemma1.1}.
\end{proof}

From Lemma \ref{lemma1.1}, we obtain 
\[ 
|\nabla\phi(x)| \le C \int_M \frac{|W(y)|}{d(x,y)^{n-1}} \,dg(y) = C \left( (I_1 W)(x)\right) . 
\] 
Therefore, we have
\[
|\nabla\phi|\in L^p(M),\qquad
p\in(n,\frac{2\beta}{2-\beta}).
\]
The proof of Theorem \ref{th:1.1} part (b) is complete.

Next, we prove Theorem \ref{th:1.1} (c).

\noindent{\it Step 3.} The integral Laplace comparison for the conformal distance $\tilde r$.

Recalling the expressions for \(\Delta_{\tilde{g}}u\), \(\nabla_{\tilde{g}}u\) and \(\operatorname{Hess}_{\tilde{g}}u\), and substituting them into \eqref{eq:bochner_mod1}, we have
\be
\lab{modcdineq2}
\al
&\frac{1}{2} e^{2 f} \Delta ( |\nabla _{\tilde  g} u |^2_{\tilde  g})  \\
&\ge   \sum^{n}_{i, j=1} | (Hess_{\tilde  g}  \, u)_{i j}+ A_{ij}|^2
+ h \left\langle  \nabla _{\tilde  g} \left( h \Delta_{\tilde  g} u - \frac{n-2}{2} \left\langle \nabla _{\tilde  g} h, \nabla _{\tilde  g} u\right\rangle _{\tilde  g} \right), \nabla _{\tilde  g} u \right\rangle _{\tilde  g}  \\
&\quad - 2 h \left\langle \nabla _{\tilde  g} (h |\nabla _{\tilde  g} u |^2_{\tilde  g}), \nabla _{\tilde  g} f \right\rangle _{\tilde  g} -  \rho_0 h | \nabla _{\tilde  g} u|^2_{\tilde  g},
\eal
\ee
where $A_{ij}=\frac{1}{2} h^{-1} g^{kl} (\pa_i h g_{jl}+ \pa_j h g_{li}- \pa_l h g_{ij}) \pa_k u$.

Furthermore, we take the function $u$ to be the conformal distance function, i.e.
\begin{equation}
\label{u}
u=\tilde r = d_{\tilde g}(p_0,x),
\end{equation}
where $p_0\in M$ is a fixed point. 
Combining \eqref{u} and \eqref{modcdineq2} with the fact that \(|\nabla_{\tilde g}\tilde r|_{\tilde g}=1\), 
and estimating the term \(\sum^{n}_{i,j=1} |(\operatorname{Hess}_{\tilde g}u)_{ij}+A_{ij}|^2\) (see \cite[pp. 18-19]{TZZZZ}),
we derive
\be
\lab{eq:bochner_mod3}
\al
0 \ge \frac{1}{n-1} \left(\Delta_{\tilde g} \tilde r + \Phi \right)^2
+ h^{-1} \pa_{\tilde r} \left(h \Delta_{\tilde g} \tilde  r - \frac{n-2}{2} \partial_{\tilde r} h \right) - h^{-2} | \nabla_{\tilde g} h |^2_{\tilde g} -  \rho_0 h^{-1}, 
\eal
\ee 
where
\[
\Phi \equiv h^{-1} \sum^{n-1}_{i=1} A_{ii},\qquad
A_{ii}= \frac{1}{2} h^{-1} g^{kl} (\partial_i h g_{il}+ \partial_i h g_{li}- \partial_l h g_{ii}) \partial_k \tilde r,
\]
so that
\[
|\Phi| \le  c \, h^{-1} |\nabla h | \, |\nabla \tilde r| = c \, |\nabla_{\tilde g} h |_{\tilde g} \, |\nabla_{\tilde g} \tilde r|_{\tilde g}= c \, |\nabla_{\tilde g} h |_{\tilde g}.
\] 
Here $c$ is a dimensional constant. 

Let
$
\psi:=\Delta_{\tilde g}\tilde r-\frac{n-1}{\tilde r}.
$
Expanding \eqref{eq:bochner_mod3} and carrying out a series of calculations (see \cite[pp. 19-20]{TZZZZ}),
we obtain a Riccati-type differential inequality for $\psi$, i.e.
\be
\lab{ricatpsi}
\al
&\frac{1}{n-1} \psi^2 +  \partial_{\tilde r} (\psi - H)+ \frac{2 \psi}{\tilde r} + \left(\frac{2 \Phi }{n-1}  + \frac{2 H}{n-2} \right) \psi + \left( 2 \Phi + \frac{2(n-1) H}{n-2} \right) \frac{1}{\tilde r}  \\
&\le \frac{2 H^2}{n-2} + h^{-2} | \nabla_{\tilde g} h |^2_{\tilde g} +  \rho_0 h^{-1}.
\eal
\ee
Here $H:=\frac{n-2}{2}h^{-1}\partial_{\tilde r}h.$
The Riccati-type inequality prepares the ground for the integral estimates in the next step,
we now proceed with the proof below.

\noindent{\it Step 4.}  $L^q$ norm estimate for $\psi_+$.

Follow the method of Petersen and Wei \cite{PW}.
Choose
$
q\in(n-2,\,\frac{4n\beta}{2n-(n-2)\beta}-2).
$
Multiplying \eqref{ricatpsi} by $(\psi-H)_+^q\sqrt{\det\tilde g}$, 
and integrating along a geodesic from \(p_0\) to the cut-locus,
we deduce
\be
\label{psint1}
\al
& \frac{1}{n-1} \int^{r_*}_0 \psi^2 (\psi-H)^q_+ \sqrt{\det \tilde g}  d\tilde r + \underbrace{ \int^{r_*}_0 (\partial_{\tilde r} (\psi - H) )(\psi-H)^q_+ \sqrt{\det \tilde g}  d\tilde r}_{\equiv L_2} \\
&\qquad + \underbrace{\int^{r_*}_0 \frac{2 \psi}{\tilde r} (\psi-H)^q_+ \sqrt{\det \tilde g}  d\tilde r }_{\equiv L_3} + \int^{r_*}_0 \left(\frac{2 \Phi }{n-1}  + \frac{2 H}{n-2} \right) \psi  (\psi-H)^q_+ \sqrt{\det \tilde g}  d\tilde r \\
&\qquad + \int^{r_*}_0 \left( 2 \Phi + \frac{2(n-1) H}{n-2} \right) \frac{1}{\tilde r}  (\psi-H)^q_+ \sqrt{\det \tilde g}  d\tilde r \\
&\le \int^{r_*}_0 \left( \frac{2 H^2}{n-2} + h^{-2} | \nabla_{\tilde g} h |^2_{\tilde g} +  \rho_0 h^{-1} \right) (\psi-H)^q_+ \sqrt{\det \tilde g}  d\tilde r.
\eal
\ee
Here $r_*$ is the length of the geodesic from $p_0$ to the cut-locus, which also depends on the angular part of the geodesic regarded as the image of a line segment under the exponential map. 

We first estimate the worst term $L_2$.
Using integration by parts and the standard formula 
$
\partial_{\tilde r}\ln \sqrt{\det \tilde g}=\Delta_{\tilde g}\tilde r
$
outside the cut-locus,
together with the identity
$
\Delta_{\tilde g} \tilde r = \psi + \frac{n-1}{\tilde r},
$
we conclude that
\begin{equation*} \label{L_2}
	\begin{aligned}
		L_2 &\ge -\frac{1}{q+1}\int_0^{r_*} (\psi-H)_+^{q+1}\left(\psi+\frac{n-1}{\tilde r}\right)\sqrt{\det\tilde g}\,d\tilde r\\
		&=-\frac{1}{q+1}\int_0^{r_*} (\psi-H)_+^{q+1}\psi\sqrt{\det\tilde g}\,d\tilde r
		-\frac{n-1}{q+1}\int_0^{r_*}(\psi-H)_+^{q+1}\frac{1}{\tilde r}\sqrt{\det\tilde g}\,d\tilde r.
	\end{aligned}
\end{equation*}
Using $L_3$ to cancel the term $\frac{1}{\tilde r}$ in $L_2$, this implies that
\begin{equation} \label{L_3-}
\al
&L_3 - \frac{n-1}{q+1} \int^{r_*}_0   (\psi-H)^{q+1} _+  \frac{1}{\tilde  r} \sqrt{det \tilde  g} d\tilde  r \\
&=\int^{r_*}_0 \frac{2 \psi}{\tilde  r} (\psi-H)^q_+ \sqrt{det \tilde  g} d\tilde  r -\frac{n-1}{q+1} \int^{r_*}_0   (\psi-H)^{q+1} _+  \frac{1}{\tilde  r} \sqrt{det \tilde  g} d\tilde  r \\
&=\int^{r_*}_0 \frac{2 (\psi-H)}{\tilde  r} (\psi-H)^q_+ \sqrt{det \tilde  g} d\tilde  r -\frac{n-1}{q+1} \int^{r_*}_0   (\psi-H)^{q+1} _+  \frac{1}{\tilde  r} \sqrt{det \tilde  g}
d\tilde  r \\
&\qquad +\int^{r_*}_0 \frac{2 H}{\tilde  r} (\psi-H)^q_+ \sqrt{det \tilde  g} d\tilde  r \\
&=\left(2-\frac{n-1}{q+1} \right) \int^{r_*}_0  (\psi-H)^{q+1} _+  \frac{1}{\tilde  r} \sqrt{det \tilde  g}
d\tilde  r
+\int^{r_*}_0 \frac{2 H}{\tilde  r} (\psi-H)^q_+ \sqrt{det \tilde  g} d\tilde  r \\
&
\ge \int^{r_*}_0 \frac{2 H}{\tilde  r} (\psi-H)^q_+ \sqrt{det \tilde  g} d\tilde  r, 
\eal
\end{equation}
The crucial point is that $q>n-2$. 

Combining \eqref{psint1} with \eqref{L_3-}, and recalling that in geodesic polar coordinates $(r,\theta)$ with $\theta \in \mathbb{S}^{n-1}$, the Jacobian factor is given by $\sqrt{\det \tilde g}=\sqrt{\det \tilde g(r,\theta)}$, we expand $\psi=(\psi-H)+H$ (see \cite[p. 22]{TZZZZ}) and obtain
\be
\lab{psint3.1}
\al
& \left(\frac{1}{n-1} - \frac{1}{q+1} \right) \int_M  (\psi-H)^{q+2}_+  d\tilde g \\
&\le \int_M \underbrace{\left(\frac{1}{q+1}- \frac{2 \Phi +2 H }{n-1}  - \frac{2 H}{n-2} \right)}_{f_1}   (\psi-H)^{q+1}_+  d\tilde g  \\
&- \int_M \underbrace{\left( 2 \Phi + \frac{2(n-1) H}{n-2} + 2 H \right)}_{f_2} \frac{1}{\tilde r}  (\psi-H)^q_+  d\tilde g  \\
&\qquad + \int_M \underbrace{\left[ \left( \frac{2 H^2}{n-2} + h^{-2} | \nabla_{\tilde g} h |^2_{\tilde g} +  \rho_0 h^{-1} \right) - \left(\frac{2 \Phi }{n-1}  + \frac{2 H}{n-2} \right) H \right]}_{f_3} (\psi-H)^q_+  d\tilde g\\
&\equiv R_1+ R_2 + R_3,
\eal
\ee 

The argument is similar to that in \cite[pp. 23-24]{TZZZZ}.
From Step 2, we have $\|\Phi\|_{L^p}\le \tilde b_p$ and $\|H\|_{L^p}\le \tilde b_p$. Hence the $L^p$ norms of $f_1,f_2,f_3$ are bounded.
By H\"older's inequality, the three error terms are bounded as follows
\[
\al
&|R_1| \le \Vert (\psi-H)_+ \Vert^{q+1}_{L^{q+2}} \, \Vert f_1 \Vert_{L^{q+2}} \le
\tilde b_p \Vert (\psi-H)_+ \Vert^{q+1}_{L^{q+2}}.\\
&|R_3| \le \Vert (\psi-H)_+ \Vert^{q}_{L^{q+2}} \, \Vert f_3 \Vert_{L^{\frac{q+2}{2}}} \le
\tilde b_p \Vert (\psi-H)_+ \Vert^{q}_{L^{q+2}}.
\eal
\]
\[
\al
|R_2| &\le \left( \int_M (\psi-H)^{q+2}_+ d \tilde g \right)^{\frac{q}{q+2}} \,
\left( \int_M \frac{|f_2|^{\frac{q+2}{2}}}{\tilde r^{\frac{q+2}{2}}} d \tilde g \right)^{\frac{2}{q+2}}\\
&\le \Vert (\psi-H)_+ \Vert^{q}_{L^{q+2}} \, \Vert 1/\tilde r \Vert_{L^{\frac{(q+2)(1+\e)}{2}} }\, \Vert f_2 \Vert_{L^{\frac{(q+2)(1+\e)}{2\e}}}.
\eal
\]
The condition $q+2<\frac{4n\beta}{2n-(n-2)\beta}$ guarantees that one can choose $\epsilon>0$ such that
$\frac{(q+2)(1+\e)}{2}<n$ and $\frac{(q+2)(1+\e)}{2 \e}<\frac{2\beta}{2-\beta}.$
Consequently, the \(L^p\) norm of $\frac{1}{\tilde r}$ is bounded, i.e. $|R_2| \le \tilde b_p \Vert (\psi-H)_+ \Vert^{q}_{L^{q+2}}.$
Then the sum  $R_1+R_2+R_3$ gives
\[
\int_{\{ \psi \ge 0 \}}   (\psi-H)^{q+2}_+  d\tilde  g \le \int_M  (\psi-H)^{q+2}_+  d\tilde  g  \le \tilde  b_q.
\]
It remains to recover the estimate for $\psi_+$ from that of $(\psi-H)_+$. 
Since $\psi=(\psi-H)+H$, we distinguish two cases.

\textit{Case 1:} $\psi(x)\ge0$ and $\psi(x)\ge H(x)$. Then
\[
\psi^{q+2}(x) = (\psi(x)-H(x) + H(x))^{q+2} \le 2^{q+2} \left[ (\psi(x)-H(x))^{q+2}_+
+|H(x)|^{q+2} \right].
\]

\textit{Case 2:} $\psi(x)\ge0$ but $\psi(x)<H(x)$. Then
\[
\psi^{q+2}(x) \le
H^{q+2}(x) =|H(x)|^{q+2}.
\]
Since $H\in L^{q+2}$ (by the gradient estimates for $h$), we conclude that
\[
\int_M \psi_+^{q+2}\,d\tilde g \le \tilde  b_q.
\]
Renaming $q+2$ as $q$, we obtain
\[
\|\psi_+\|_{L^q(M,\tilde g)}\le H_q^*, \qquad q\in(n,\,\frac{4n\beta}{2n-(n-2)\beta}).
\]
This completes the full proof of Theorem \ref{th:1.1}.
\qed
\medskip

Next, we turn to 

\textbf{\emph{Proof of Corollary \ref{volcomp}.}} 
(a). Let $\tau_+ := (\Delta r - \frac{n-1}{r})_+$. The following well-known inequality holds on $(M, g)$ (see e.g. \cite{PW})
\be
\label{differential of volume ratio}
\frac{d}{dr}\left(\frac{\operatorname{Vol}_{ g}(B(x,r))}{r^n}\right)\leq \frac{1}{r^{n+1}}\int_{\mb{S}^{n-1}}\int_0^r s\,\tau_+ \, \sqrt{det g(s,\theta)} dsd\theta.
\ee 
Recall that
$
\psi_+ =\left(  \Delta_{\tilde g} \tilde r - \frac{n-1}{\tilde r}\right) _+.
$
Combining \eqref{differential of volume ratio} for the metric $\tilde g$ with the bound for $\psi_+$ in Theorem \ref{th:1.1},  we obtain
\[
\al
\frac{d}{dr}\left(\frac{\operatorname{Vol}_{\tilde g}(B_{\tilde g}(x,r))}{r^n}\right) 
&\leq \frac{1}{r^n} \int_{B_{\tilde g}(x, r)} \psi_+ \, d \tilde g
\le \frac{1}{r^n} \Vert \psi_+ \Vert_{L^q(M, \tilde g)}
\, \operatorname{Vol}_{\tilde g}(B_{\tilde g}(x, r))^{\frac{q-1}{q}} \\
&\le \frac{H_q^*}{r^n}
\, (Cr^n)^{\frac{q-1}{q}} \le \frac{\tilde b_q}{r^{\frac{n}{q}}}.
\eal
\]
for $q\in(n,\,\frac{4n\beta}{2n-(n-2)\beta})$. Here we have renamed the constant $\tilde b_q$. Likewise,
\[
\frac{\operatorname{Vol}_{\tilde g} (\partial B_{\tilde g}(x, r_2))}{r^{n-1}_2} -
\frac{\operatorname{Vol}_{\tilde g}(\partial B_{\tilde g}(x, r_1))}{r^{n-1}_1}
\le \int_{B_{\tilde g}(x, r_2)\setminus B_{\tilde g}(x, r_1)}  \frac{\psi_+}{r^{n-1}} d\tilde g
\le \tilde b_q \left(r^{\frac{q-n}{q-1}}_2 - r^{\frac{q-n}{q-1}}_1 \right)^\frac{q-1}{q}.
\]

(b). Following \cite[proof of Corollary 1.2]{TZZZZ},
we note that the H\"older continuity of $h$ and Property 2 remain valid in the present setting. 
Therefore, the same proof applies and yields
\be \label{AA}
\frac{\operatorname{Vol}(B(x,r_2))}{r_2^n} - \frac{\operatorname{Vol}(B(x,r_1))}{r_1^n}
\leq b_q r_2^{1-\frac{n}{q}}.
\ee
Since $(M, g)$ satisfies the volume doubling condition by Property 1 and 2, or \eqref{AA},
we can use Lemma 3.3 in \cite{Colding},
which infers
\begin{align*}
\frac{\operatorname{Vol}(B(x,r_2))}{r_2^n} - \frac{\operatorname{Vol}(B(x,r_1))}{r_1^n}
&\leq C \left( 1-\frac{r_1}{r_2} \right)^\delta \frac{\operatorname{Vol}(B(x,r_2))}{r_2^n}  \\
&\leq C \left( 1-\frac{r_1}{r_2} \right)^\delta.
\end{align*}
Here $C$ and $\delta$ depend only on the doubling constant, which depend only on $q$ and the basic parameters $C_0$, $\beta$, $\operatorname{diam} M$, $C_S$ and $n$ in Theorem \ref{th:1.1}.
This completes the proof of Corollary \ref{volcomp}.
\qed

\section*{Acknowledgment}

I am grateful to Professors Jiancheng Liu, Qi S. Zhang and Yu Zheng for their guidance  and encouragement, 
and to PhD students Juanling Lu and Xuenan Fu for their helpful discussions.
I also thank Christian Rose for pointing out a reference in the introduction.
This research was supported by NSF of China (No. 12161078) and the Funds for Innovative Fundamental Research Group Project of Gansu Province (No. 24JRRA778).

During the exploratory stage of this work, 
ChatGPT drew my attention to the article by Adams \cite{Adams}
and also assisted with editing and polishing the presentation and wording of the manuscript.
All mathematical arguments were independently developed by the author.



\end{document}